\documentclass[reqno,11pt]{amsart}

\usepackage[T2A]{fontenc}
\usepackage[utf8]{inputenc}
\usepackage[english]{babel}
\usepackage{amsmath}
\usepackage{amssymb}
\usepackage{amsfonts}
\usepackage{graphicx}
\usepackage{hyperref}
\usepackage{setspace}
\usepackage{csquotes}

\usepackage[margin=1in,includefoot,footskip=30pt]{geometry}

\usepackage{float}
\usepackage{amsthm}
\usepackage{multicol}
\usepackage{tikz-cd}
\usetikzlibrary{decorations.pathmorphing}
\usepackage{mathtools}
\usepackage{enumitem}
\usepackage{ytableau}

\DeclareGraphicsExtensions{.png,.jpg}
\ytableausetup{centertableaux}

\newcommand{\Z}{\mathbb{Z}}
\newcommand{\Complex }{\mathbb{C}}
\newcommand{\Sym}{\text{Sym}}
\newcommand{\sgn}{\text{sgn}}

\newcommand{\Symgrp}[1]{{\mathfrak{S}_{#1}}}
\newcommand{\g}[1]{{\mathfrak{g}_{#1}}}

\newcommand{\gl}[1]{{\mathfrak{gl}_{#1}\mathbb{C}}}
\newcommand{\GL}[1]{{\text{GL}_{#1}\mathbb{C}}}
\newcommand{\GZ}[1]{{\text{GZ}_{#1}}}

\newcommand{\gli}{{\mathfrak{gl}_{\infty}\mathbb{C}}}
\newcommand{\GLi}{{\text{GL}_{\infty}\mathbb{C}}}
\newcommand{\GZi}{{\text{GZ}_{\infty}}}

\hypersetup{
    colorlinks = true,
    linkcolor = blue,
    filecolor = magenta,
    urlcolor = blue,
    citecolor = green
}

\numberwithin{equation}{section}
\theoremstyle{plain}
\newtheorem{lemma}{Lemma}[section]
\newtheorem{theorem}{Theorem}[section]
\newtheorem{proposition}{Proposition}[section]

\begin{document}
    \title[The Gelfand--Tsetlin basis for irreducible $\GLi$-modules]{The Gelfand--Tsetlin basis for irreducible representations of an infinite-dimensional general linear group}
    \author[E.~Movchan]{Evgenii Movchan}
    \email[Evgenii Movchan]{evgenymov@gmail.com}

    \begin{abstract}
        We consider the problem of constructing a Gelfand--Tsetlin basis in irreducible representations of an infinite-dimensional general linear group. For a finite-dimensional irreducible representation of a general linear group, all elements of the Gelfand--Tsetlin basis are parameterized by Gelfand--Tsetlin schemes. We extend this definition to infinite Gelfand--Tsetlin schemes, which in turn parameterize elements of the Gelfand--Tsetlin basis of an irreducible representation of an infinite-dimensional complete linear group. Using the properties of colimits of representations with the highest weight, we present an explicit form of the Gelfand--Tsetlin basis.
    \end{abstract}

    \maketitle
    
    \pagenumbering{arabic}

    \begin{sloppypar}
        \section{Introduction}
        Let $\{ \g{j} \}_{j = 0} ^n$ be a family of finite-dimensional reductive complex Lie algebras, such, that $\g{0} = 0$. For this family, let's consider a chain of embeddings
\begin{equation}
    \begin{tikzcd}[column sep = 4em, row sep = 2em]
        \g{0} \arrow[hook]{r}{\iota_0} & \g{1} \arrow[hook]{r}{\iota_1} & \g{2} \arrow[hook]{r}{\iota_2} & \cdots \arrow[hook]{r}{\iota_{n-1}} & \g{n}.
    \end{tikzcd}
    \label{equation:1.1}
\end{equation}

According to Weyl's theorem, any finite-dimensional complex representation $V$ of the algebra $\g{n}$ is completely reducible. In other words, it has a unique decomposition up to isomorphism

\begin{equation}
    V \cong \bigoplus \limits_\lambda {V^\lambda}^ {\oplus a_\lambda},
    \label{equation:1.2}
\end{equation}
where $V^\lambda$ are irreducible finite-dimensional complex representations of the algebra $\g{n}$, and $a_\lambda$ are some natural numbers.

We will call the numbers $a_\lambda$ the multiplicities of occurrence of the representation $V^\lambda$ in $V$. Any irreducible finite-dimensional complex representation $V^\lambda$ of the algebra $\g{n}$ is uniquely (up to isomorphism) determined by its highest weight $\lambda$. Using a chain of embeddings (\ref{equation:1.1}) for each such representation $V^\lambda$, considering its decomposition $V^\lambda \cong \bigoplus \limits_T V_T$ into irreducible $\g{0}$-modules and choosing in each one-dimensional subspace $V_T = \langle e_T \rangle$ a basis vector $e_T$, we can construct a natural basis $\{e_T\} \subset V^\lambda$, called the Gelfand-Tsetlin basis. The canonicity of this basis depends on the multiplicities $a_\lambda$ in the expansion (\ref{equation:1.2}). In particular, for the case of interest to us $\g{n} = \gl{n}$, all multiplicities $a_\lambda \equiv 1$ (see \cite{fulton1997young} ch.5). This basis was first constructed by Gelfand and Tsetlin in their papers \cite{cetlin1950finite, gelfand1950finite, gel1965finite}. Such natural basis turned out to be very useful both for pure mathematics and theoretical physics.

This theory was later developed by Zhelobenko, and in 1962, a method for constructing lowering operators for representations $\gl{n}$ was described in the work \cite{zhelobenko1962classical} (reducing operators first appeared in the work \cite{nagel1965operators}). Using this lowering operators, it was possible to construct a Gelfand--Tsetlin basis in all finite-dimensional complex representations of the classical Lie algebras $A_n$, $B_n$, $C_n$, $D_n$, which is described, for example, in the works \cite{molev2006, molev2021monomial}. The key property of these operators is the following: let $V^\lambda$ be an irreducible finite-dimensional complex representation of the reductive Lie algebra $\g{n}$. By Weyl's theorem, there is decomposition of $\text{Res}^\g{n}_\g{n-1} V^\lambda \cong \bigoplus \limits_\mu {V^\mu}^ {\oplus a_\mu}$ to irreducible $\g{n-1}$-modules. Reducing operators $z_{n i} \in U(\g{n})$, with index $1 \leq i \leq n-1$, define maps (under certain conditions on the weights)
\begin{equation}
    z_{n i}: V^\lambda \longrightarrow V^\lambda: z_{n i} \cdot v_\lambda = v_{\lambda - \delta_i},
    \label{equation:1.3}
\end{equation}
where $\lambda - \delta_i$ is the highest weight obtained from $\lambda$ by replacing the term $\lambda_i$ with $\lambda_{i+1}$, and $v_\lambda$ and $v_{\lambda - \delta_i}$ are highest weight vectors, respectively. By acting with these operators with different indices $i$ on the highest weight vector $v_\lambda$ of the representation $V^\lambda$, the entire Gelfand--Tsetlin basis of this representation can be obtained.

We will be interested in the algebra $\gli$, for which it is not clear what the phrase “irreducible $\mathfrak{gl}_{\infty - 1}\mathbb{C}$-module” means. In this case, the theory of lowering operators described above does not directly suit us. The observation inspired by the Yangian theory turns out to be productive (see \cite{molev2009yangians}). As shown in the works of Nazarov and Tarasov \cite{nazarov1994yangians, nazarov1998representations}, for a fixed algebra $\g{n}$ and a representation $V^\lambda$ there are operators $\{A_m (u)\}_{m = 1} ^n$, called quantum minors of the $L$-operator, such that the Gelfand--Tsetlin basis is their eigenvectors. In other words, all operators $\{A_m (u)\}_{m = 1} ^n$ act diagonally in the Gelfand--Tsetlin basis. Thus, knowing the operators $\{A_m(u)\}_{m=1}^n$ and their eigenvalues $\{\lambda_m(u)\}_{m=1}^n$, we can reduce the problem of constructing the Gelfand--Tsetlin basis to a spectral problem. Such an approach is convenient when the action of lowering operators on the highest weight vector is not defined in principle. This approach was used in the works \cite{valinevich2019construction, antonenko2022gelfand}, where the elements of the Gelfand--Tsetlin basis were directly defined as eigenvectors of quantum minors $\{A_m(u)\}_{m=1}^n$. So, to construct the Gelfand--Tsetlin basis in irreducible representations of the algebra $\gli$, together with the well-known theory of lowering operators $z_{n i}$, we need the theory of quantum minors $A_m(u)$. Using these two approaches, it will be possible to reduce the problem to a finite-dimensional case. More information about these operators and the correctness of the objects they define will be written below.

This work has the following structure. In section 2, we note the general facts of the representation theory of the Lie algebra $\gl{n}$. In section 3, using the combinatorial Gelfand--Tsetlin schemes, an explicit formula of the Gelfand--Tsetlin basis in irreducible representations of the algebra $\gl{n}$ is given. Next, the infinite-dimensional group $\GLi$ and the algebra $\gli$ are formally defined, and the equivalence of their representations is noted. In section 4, an important concept of the Gelfand--Tsetlin algebra $\GZ{n}$ is introduced, with the help of which an alternative definition of the Gelfand--Tsetlin basis is given as a basis, in which the algebra $\GZ{n}$ acts diagonally. Finally, we formally define the infinite Gelfand--Tsetlin algebra $\GZi$. In section 5, the polynomial representations of the algebra $\gli$ are defined as solutions to a universal problem similar to the finite-dimensional case. Their existence and correct certainty are shown. In section 6, containing the main original results of this work, by analogy with the finite-dimensional case, we define infinite Gelfand--Tsetlin schemes, with the help of which we construct the Gelfand--Tsetlin basis in polynomial representations of the algebra $\gli$. In parallel, it is proved that the polynomial representations of the algebra $\gli$ are irreducible representations with highest weight. In the end, we note that the ideas used can be applied to construct the Gelfand--Tsetlin basis in irreducible representations with highest weight of the colimits of all classical Lie algebras.
    
        \section{Representation theory of general linear Lie algebra}
        As is known, all irreducible finite-dimensional complex representations (more precisely, their isomorphism classes) of the algebra $\gl{n}$ are in one-to-one correspondence with ordered sets $\lambda=(\lambda_1,\ldots,\lambda_n)$, called highest weights, such that
\begin{equation}
    \lambda_1 \geq \ldots \geq \lambda_n \qquad \text{and} \qquad  \forall i \in \{ 1, \ldots, n \} \implies \lambda_i \in \Z.
        \label{equation:2.1}
\end{equation}

Denote by $V^\lambda$ the irreducible $\gl{n}$-module indexed by the highest weight $\lambda$. Let $\lambda_m \geq 0$ and $\lambda_{m+1} \leq 0$, define partitions $\mu =\mu_1 \geq \ldots\geq \mu_m$ and $\nu =\nu_{m+1} \geq\ldots \geq\nu_n$  of the number $n$, such that $\lambda_1 = \mu_1, \ldots, \lambda_m =\mu_m$ and $\lambda_{m+1} = -\nu_{m+1}, \ldots,\lambda_{n} = -\nu_{n}$. Let's identify partitions and Young diagrams. Thus, any irreducible $\gl{n}$-module $V^\lambda$ is parameterized by a pair of Young diagrams $V^{\mu\nu} = V^\lambda$.

For $k\geq 0$, we define the determinant representation $D_k = (\bigwedge\nolimits^n\Complex^n)^{\otimes k}$ as the $k$-th tensor degree of the $n$-th wedge power of the tautological representation $\Complex^n$ of the algebra $\gl{n}$. For $k<0$, we define the determinant representation $D_k = D_{-k}^*$ as dual to $D_{-k}$. Using the determinant representation, it is not difficult to prove the following isomorphism of representations:
\begin{equation}
    V^{\mu \nu} \cong V^{\widetilde{\mu} \widetilde{\nu}} \otimes D_{-k} \qquad \text{for} \qquad  k - \nu_{m+1}\geq 0,
    \label{equation:2.2}
\end{equation}
where $\widetilde{\mu} = \lambda_1 + k \geq \ldots \geq \lambda_m + k$ and $\widetilde{\nu} = 0 \geq \ldots \geq 0 = 0$.

Representations $V^{\mu\nu}$, whose partitions $\nu=0$, are called polynomial representations. As can be seen from (\ref{equation:2.2}), any irreducible representation $V^{\mu\nu}$ is isomorphic to the tensor product of the polynomial and determinant representations. We will be interested in the basis in the representations $V^{\mu\nu}$, and the key observation here is that all determinant representations are one-dimensional, that is, $\text{dim} D_k = 1$ for all $k$. Thus, without loss of generality, for an arbitrary irreducible finite-dimensional $\gl{n}$-module $V^{\lambda}$, it can be assumed that $\lambda$ is a Young diagram (which is further always assumed). This observation is the motivation to consider only the irreducible polynomial representations of the algebra $\gli$.

For the algebra $\gl{n}$ and fixed Young diagram $\lambda$ of $m$ cells, we define the representation $(\Complex^n)^{\times\lambda}$ as the Cartesian product of $m$ copies of the tautological representation $\Complex^n$, indexed by the cells of the diagram $\lambda$ (see \cite{fulton1997young} ch.8). Thus, the elements of the representation $(\Complex^n)^{\times\lambda}$ are diagrams $\lambda$, in each cell of which an element from $\Complex^n$ is written. For example, for $\lambda = (2,2,1)$, an arbitrary element $w\in ( \Complex^n )^{\times\lambda}$ has the form
\begin{equation}
\ytableausetup{mathmode, boxframe=normal, boxsize=2em}
    w = 
    \begin{ytableau}
        v_1 & v_2 \\
        v_3 & v_4 \\
        v_5
    \end{ytableau} \text{~},
    \label{equation:2.3}
\end{equation}
where $v_i$ are some elements from $\Complex^n$.

Let's call the map $f: ( \Complex^n )^{\times\lambda}\xrightarrow{\quad} F$ into some vector space $F$ symmetrizing if it satisfies the conditions:
\begin{align}
    &\text{1) $f$ -- multilinear}, \label{equation:2.4} \\
    &\text{2) $f$ -- skew-symmetric across elements in the same column}, \label{equation:2.5} \\
    &\text{3) $\forall w \in ( \Complex^n )^{\times \lambda} \implies f(w) = \sum\limits_{w'}f(w')$}, \label{equation:2.6}
\end{align}
where summation is implied by those $w'\in ( \Complex^n )^{\times\lambda}$ resulting from the $w$ from exchange between two fixed columns, with the selected subset of cells in the right of the selected columns. For example, for $\lambda = (2,2,1)$, selecting the entire right column, we get
\begin{equation}
    f\Biggl( \text{~}
    \begin{ytableau}
        v_1 & v_2 \\
        v_3 & v_4 \\
        v_5
    \end{ytableau} \text{~} \Biggl)
    = 
    f\Biggl( \text{~}
    \begin{ytableau}
        v_2 & v_1 \\
        v_4 & v_3 \\
        v_5
    \end{ytableau} \text{~} \Biggl)
    +
    f\Biggl( \text{~}
    \begin{ytableau}
        v_1 & v_3 \\
        v_2 & v_5 \\
        v_4
    \end{ytableau} \text{~} \Biggl)
    +
    f\Biggl( \text{~}
    \begin{ytableau}
        v_2 & v_1 \\
        v_3 & v_5 \\
        v_4
    \end{ytableau} \text{~} \Biggl).
    \label{equation:2.7}
\end{equation}

The polynomial representations $V^\lambda$ of the algebra $\gl{n}$ are solutions to the following universal problem: if $\Phi_{uni} : (\Complex^n)^{\times\lambda}\xrightarrow{\quad} V^\lambda$ is a symmetrizing map, such that for any symmetrizing map $\Phi:(\Complex^n)^{\times\lambda}\xrightarrow{\quad}F$, there is a single map $\varphi : V^\lambda\xrightarrow{\quad}F$, such that $\Phi=\varphi\circ \Phi_{uni}$. In other words, for any $\Phi$ and $F$ as above, there is a single map $\varphi$, that makes the following diagram commutative
\begin{equation}
    \begin{tikzcd}[column sep = 4em, row sep = 4em]
        ( \Complex^n )^{\times \lambda} \arrow{r}{\Phi_{uni}} \arrow[swap]{dr}{\Phi} & V^\lambda \arrow[dashrightarrow]{d}{\varphi} & V^\lambda \arrow{d}{\rho}\\
        & F, & \mathbb{S}_\lambda ( \Complex^n ),
    \end{tikzcd}
    \label{equation:2.8}
\end{equation}
where $\rho : V^\lambda \xrightarrow{\quad} \mathbb{S}_\lambda ( \Complex^n )$ is an isomorphism of representations, and $\mathbb{S}_\lambda ( \Complex^n )$ is a Schur functor.
        
        \section{An infinite-dimensional general linear group and its Lie algebra}
        As mentioned in the introduction, the Gelfand--Tsetlin basis was constructed by Gelfand and Tsetlin for all irreducible finite-dimensional $\gl{n}$-modules $V^\lambda$. The elements of this basis are naturally parameterized by combinatorial objects called Gelfand--Tsetlin schemes. For a fixed Young diagram $\lambda = \lambda_1\geq\ldots \geq\lambda_n$ the Gelfand--Tsetlin scheme is an ordered set of elements $\lambda_{i j}$
\begin{equation}
    \Lambda = \left[
    \begin{array}{ccccccccc}
    \lambda_{n 1} & & \lambda_{n 2} & & \cdots & & \cdots & & \lambda_{n n} \\
    & \lambda_{n-1,1} & & \lambda_{n-1,2} & & \cdots & & \lambda_{n-1, n-1} & \\
    & & \cdots & & \cdots & & \cdots & & \\
    & & & \lambda_{2 1} & & \lambda_{2 2} & & & \\
    & & & & \lambda_{1 1} & & & & \\
    \end{array}
    \right],
    \label{equation:3.1}
\end{equation}
where for all $2 \leq j\leq n$ and $1 \leq i \leq j - 1$ the following relations hold
\begin{equation}
    \lambda_{i j} \in \mathbb{N}, \qquad \lambda_{n j} = \lambda_j, \qquad \lambda_{i j} \geq \lambda_{i - 1, j} \qquad \text{and} \qquad \lambda_{i - 1, j} \geq \lambda_{i, j + 1}.
    \label{equation:3.2}
\end{equation}

As it was said, the set of $\Lambda_n (\lambda)$ of all Gelfand--Tsetlin schemes of the form $\lambda$ of the algebra $\gl{n}$ sets the natural parametrization of the Gelfand--Tsetlin basis in the irreducible $\gl{n}$-module $V^\lambda$. Thus, for each Gelfand--Tsetlin scheme $\Lambda$ is assigned an element $e_T \in V^\lambda$ of the Gelfand--Tsetlin basis, hereinafter referred to as $e_\Lambda$. Using Gelfand--Tsetlin schemes and lowering operators, it is possible to write the Gelfand--Tsetlin basis of the representation $V^\lambda$ explicitly (see~\cite{zhelobenko1962classical}). So, for a fixed Gelfand--Tsetlin scheme $\Lambda$, the element $e_\Lambda$ of the Gelfand--Tsetlin basis has the form
\begin{equation}
    e_\Lambda = \prod\limits_{2 \leq k \leq n}^{\longrightarrow} \prod\limits_{i = 1}^{k-1} z_{k i}^{\lambda_{k i} - \lambda_{k-1, i}} \cdot v_\lambda,
    \label{equation:3.3}
\end{equation}
where the terms in the ordered product are ordered according to the increasing index $k$, $v_\lambda$ is the highest weight vector of the representation $V^\lambda$, and the lowering operators $z_{k i}\in U(\gl{n})$ have the form
\begin{equation}
    z_{k i} = \sum\limits_{i < i_1 < \ldots < i_p < k} E_{i_1 i} \cdot E_{i_2 i_1} \cdot \ldots \cdot E_{i_p i_{p-1}} \cdot E_{k i_p} \cdot (E_{i i} - E_{j_1 j_1} + j_1 - i) \cdot \ldots \cdot (E_{i i} - E_{j_q j_q} + j_q - i),
    \label{equation:3.4}
\end{equation}
where the sum is calculated over all $p\in\mathbb{N}$, and the set $\{j_1, \ldots, j_q\}$ is the complement to the subset $\{i_1,\ldots, i_p\}$ in the set $\{i+1, \ldots,k-1 \}$.

Next, we will use an explicit formula for the Gelfand--Tsetlin basis (\ref{equation:3.3}) in an irreducible representation of the algebra $\gl{n}$ to construct the Gelfand--Tsetlin basis in irreducible representations of the algebra $\gli$, but for now let's define how we understand the group $\GLi$ and the algebra $\gli$. To do this, let's consider a commutative diagram
\begin{equation}
    \begin{tikzcd}[column sep = 4em, row sep = 4em]
        0 \arrow[hook]{r}{d_e \iota_0} & \gl{1} \arrow[hook]{r}{d_e \iota_1} \arrow{d}{\exp} & \gl{2} \arrow[hook]{r}{d_e \iota_2} \arrow{d}{\exp} & \cdots \arrow[hook]{r}{d_e \iota_{n-1}} & \gl{n} \arrow[hook]{r}{d_e \iota_n} \arrow{d}{\exp} & \cdots \\
        0 \arrow[hook]{r}{\iota_0} & \GL{1} \arrow[hook]{r}{\iota_1} & \GL{2} \arrow[hook]{r}{\iota_2} & \cdots \arrow[hook]{r}{\iota_{n-1}} & \GL{n} \arrow[hook]{r}{\iota_n} & \cdots,
    \end{tikzcd}
    \label{equation:3.8}
\end{equation}
where $\exp$ is an exponential map, $\iota_{n-1} : \GL{n-1} \xhookrightarrow{\quad} \GL{n}: A \mapsto
\begin{pmatrix}
A & \vline & 0 \\
\hline
0 & \vline & 1
\end{pmatrix}
$ -- injection, and $d_e\iota_{n-1}$ is its differential in identity $e$.

Colimit of the bottom chain of embeddings of the diagram (\ref{equation:3.8}) we will understand as the group $\GLi =\varinjlim\GL{n}$, and the copy of the upper chain of embeddings as the algebra $\gli=\varinjlim\gl{n}$, respectively. The group $\GLi$ and the algebra $\gli$ can be perceived as a set of infinite matrices in which almost all elements are equal to zero, and in the case of a group, almost all elements on the diagonal are equal to one. In other words, the following equalities take place
\begin{equation}
    \GLi = \{ (a_{i j})_{(i, j) \in {\mathbb{N}^*}^2} ~|~ \forall (i, j) \in {\mathbb{N}^*}^2 \exists N \in {\mathbb{N}^*}: \text{if } i+j > N \implies a_{i j} = \delta_{i j}\},
    \label{equation:3.9}
\end{equation}
\begin{equation}
    \gli = \{ (a_{i j})_{(i, j) \in {\mathbb{N}^*}^2} ~|~ \forall (i, j) \in {\mathbb{N}^*}^2 \exists N \in {\mathbb{N}^*}: \text{if } i+j > N \implies a_{i j} = 0\},
    \label{equation:3.10}
\end{equation}
where ${\mathbb{N}^*}$ is the set of natural numbers without zero, and $\delta_{i j}$ is the Kronecker symbol.

Note that usually an infinite-dimensional general linear group and the symbol $\GLi$ are understood to be a group whose elements are infinite in both directions, that is, they have the form $(a_{i j})_{(i, j)\in {\mathbb{Z}}^2}$ (see \cite{kac2013bombay} ch.4). Thus, it would be more correct for us to use the notation ${\text{GL}_{\infty/2}\mathbb{C}}$ and ${\mathfrak{gl}_{\infty/2}\mathbb{C}}$. However, ignoring the possible confusion, we will use the above notation $\GLi$ and $\gli$.

Since $\GL{n}$ is a connected Lie group, any irreducible representation $V^\lambda$ is also an irreducible representation of its Lie algebra $\gl{n}$ and vice versa, and the following diagram is commutative
\begin{equation}
    \begin{tikzcd}[column sep = 4em, row sep = 4em]
            \gl{n} \arrow{r}{d_e \rho_\lambda} \arrow[swap]{d}{\exp} & \mathfrak{gl}(V^\lambda) \arrow{d}{\exp} \\
        \GL{n} \arrow{r}{\rho_\lambda} & \text{GL}(V^\lambda),
    \end{tikzcd}
    \label{equation:3.13}
\end{equation}
where $\rho_\lambda$ is the homomorphism of the irreducible representation $V^\lambda$.

The Lie group $\GLi$ is connected as the colimit of connected Lie groups $\GL{n}$. Indeed, let's consider arbitrary elements $A, B\in\GLi$. We want to show that there is a smooth path $\gamma : [0,1]\xrightarrow{\quad}\GLi$, such that $\gamma(0)= A$ and $\gamma(1)=B$. By definition of the colimit, for the elements $A$ and $B$ there is a smooth map $\phi_n : \GL{n}\xrightarrow{\quad}\GLi$ and the elements $\widetilde{A}, \widetilde{B}\in\GL{n}$, which are preimages of the elements $A$ and $B$. In other words, $\phi_n(\widetilde{A}) = A$ and $\phi_n(\widetilde{B}) = B$. But the Lie group $\GL{n}$ is path-connected, which means that there is a smooth path $\widetilde{\gamma} : [0,1]\xrightarrow{\quad}\GL{n}$, such that $\widetilde{\gamma}(0)=\widetilde{A}$ and $\widetilde{\gamma}(1) = \widetilde{B}$. Let's define the path $\gamma : [0,1] \xrightarrow{\quad} \GLi$ as a composition $\gamma = \phi_n \circ\widetilde{\gamma}$. Thus, all elements of the Lie group $\GLi$ are connected with a smooth path, which means that the group $\GLi$ is path-connected as a smooth manifold, thus connected.

The commutativity of diagrams (\ref{equation:3.8}) and (\ref{equation:3.13}) induces the commutativity of the colimit diagram
\begin{equation}
    \begin{tikzcd}[column sep = 4em, row sep = 4em]
        \gli \arrow{r}{d_e \rho} \arrow[swap]{d}{\exp} & \mathfrak{gl}(V) \arrow{d}{\exp} \\
        \GLi \arrow{r}{\rho} & \text{GL}(V),
    \end{tikzcd}
    \label{equation:3.14}
\end{equation}
where $V$ is some irreducible representation of $\GLi$, and $\rho$ is its homomorphism. 

Thus, speaking of irreducible $\gli$ modules, without loss of generality, we can consider only irreducible $\gli$ modules.

        \section{An infinite-dimensional Gelfand--Tsetlin algebra and the quantum minors of the L-operator}
        In addition to the usual definition of the Gelfand--Tsetlin basis of the irreducible representation $V^\lambda$ of the algebra $\gl{n}$ through the restriction on $\gl{0}$-modules described in the introduction, it is possible to define the Gelfand--Tsetlin basis as the basis on which some algebra acts diagonally. We will be interested in the second definition, because it is much more convenient to work with colimits with it. An algebra acting diagonally on the Gelfand--Tsetlin basis is called the Gelfand--Tsetlin algebra $\GZ{n}$ of the algebra $\gl{n}$. In our case, the Gelfand--Tsetlin algebra $\GZ{n}$ is some subalgebra of the universal enveloping algebra $U(\gl{n})$. Consider the chain of embeddings of universal enveloping algebras $U(\gl{n})$ induced by the chain (\ref{equation:3.8}) of algebras $\gl{n}$
\begin{equation}
    \begin{tikzcd}[column sep = 3.74em, row sep = 2em]
        0 \arrow[hook]{r}{\iota^*_0} & U(\gl{1}) \arrow[hook]{r}{\iota^*_1} & U(\gl{2}) \arrow[hook]{r}{\iota^*_2} & \cdots \arrow[hook]{r}{\iota^*_{n-1}} & U(\gl{n}) \arrow[hook]{r}{\iota^*_n} & \cdots,
    \end{tikzcd}
    \label{equation:4.1}
\end{equation}
where $\iota^*_{n-1} : U(\gl{n-1}) \xhookrightarrow{\quad} U(\gl{n}): E_{i_1 j_1} \cdot \ldots \cdot E_{i_s j_s} \mapsto \iota^*_{n-1} (E_{i_1 j_1} \cdot \ldots \cdot E_{i_s j_s}) = d_e \iota_{n-1}(E_{i_1 j_1}) \cdot \ldots \cdot d_e \iota_{n-1}(E_{i_s j_s})$ -- injection.

For the algebra $\gl{n}$ of the chain (\ref{equation:4.1}) the Gelfand--Tsetlin algebra $\GZ{n}\subset U(\gl{n})$, generated by all centers of universal enveloping algebras $U(\gl{k})$, where $0\leq k\leq n$, is called the Gelfand--Tsetlin subalgebra. In other words,
\begin{equation}
    \GZ{n} = \langle Z( U(\gl{1}) ), \ldots, Z( U(\gl{n}) )\rangle,
    \label{equation:4.2}
\end{equation}
where $Z(U(\gl{k}))$ is the center of the universal enveloping $U(\gl{k})$.

As mentioned above, a remarkable property of the Gelfand--Tsetlin algebra $\GZ{n}$ is that it acts diagonally on the Gelfand--Tsetlin basis $\{e_\Lambda\}_{\lambda\in\Lambda_n(\lambda)}$ of some irreducible $\gl{n}$-module $V^\lambda$. This property is the motivation for introducing a similar definition of the Gelfand--Tsetlin basis for irreducible $\gli$ modules. Now let's introduce the definition of the $L$-operator and its quantum minors, which is important for further discussion (see~\cite{molev2009yangians}). For a fixed algebra $\gl{n}$, an $L$-operator is a formal polynomial in the parameter $u$, called the spectral parameter, given by the formula
\begin{equation}
    L(u) = u + E, \qquad \text{где} \qquad E = 
    \left(
    \begin{array}{cccc}
    E_{1 1} & E_{1 2} & \ldots & E_{1 n}\\
    E_{2 1} & E_{2 2} & \ldots & E_{2 n}\\
    \vdots & \vdots & \ddots & \vdots\\
    E_{n 1} & E_{n 2} & \ldots & E_{n n}
    \end{array}
    \right).
    \label{equation:4.3}
\end{equation}

For the $L$-operator and the index $1\leq m\leq n$, the quantum minor is a formal polynomial in the spectral parameter $u$, given by the formula
\begin{align}
    A_m (u) = \sum\limits_{\sigma \in \mathfrak{S}_m} \sgn(\sigma) L(u)_{\sigma(1) 1} \cdot \ldots \cdot L(u - m + 1)_{\sigma(m) m},
    \label{equation:4.4}
\end{align}
where $\Symgrp{m}$ is the permutation group of $m$ elements, and $\sgn(\sigma)$ is the sign of the permutation $\sigma \in \Symgrp{m}$.

As can be seen from the definition, the quantum minor $A_m(u)$ is a polynomial of degree $m$, and can be represented as $A_m(u) =u^m +a_{m 1}u^{m-1} +\ldots+a_{m m}$. The operator $A_m(u)$ is the Capelli determinant of the algebra $\gl{m}$, so its coefficients $\{a_{m i}\}_{i=1}^m$ are generators of the center of the universal enveloping $U(\gl{m})$, that is $Z(U(\gl{m})) = \langle \{ a_{m i} \}_{i=1}^m\rangle $. Therefore, the Gelfand--Tsetlin algebra $\GZ{n}$ is generated by all coefficients $\{a_{m i}\}_{1\leq i\leq m\leq n}$. So, for any index $1\leq m\leq n$, the quantum minors $A_m(u)$ act diagonally in the Gelfand--Tsetlin basis $\{e_\Lambda\}_{\lambda\in\Lambda_n(\lambda)}$ of irreducible representation $V^\lambda$ of algebra $\gl{n}$. Namely, there is a formula
\begin{equation}
    A_m (u) \cdot e_\Lambda = \prod\limits_{i=1}^{m}(u + \lambda_{m i} - i + 1) e_\Lambda.
    \label{equation:4.5}
\end{equation}

This property uniquely (up to multiplication by scalar) defines the Gelfand--Tetlin basis $\{e_\Lambda\}_{\lambda \in\Lambda_n(\lambda)}$ in the irreducible representation of $V^\lambda$. Thus, the Gelfand--Tsetlin basis can be defined as a solution to the spectral problem (\ref{equation:4.5}), without using lowering operators. This approach can be used in cases where the action of lowering operators on the highest weight vector is not defined in principle (see~\cite{valinevich2019construction, antonenko2022gelfand}).

Let's return to the chain (\ref{equation:4.1}). Note that for any $n\in\mathbb{N}^*$, the center of the universal enveloping $U(\gl{n-1})$ is embedded in the center of the universal enveloping $U(\gl{n})$, that is, $\iota^*_{n-1} ( Z( U(\gl{n-1})) ) \subset Z(U(\gl{n}))$. Thus, we obtain a well-defined chain of embeddings of Gelfand--Tsetlin algebras, shown in the following commutative diagram
\begin{equation}
    \begin{tikzcd}[column sep = 4em, row sep = 4em]
        0 \arrow[hook]{r}{\iota^*_{0}} & \GZ{1} \arrow[hook]{r}{\iota^*_{1}} \arrow[hook]{d}{\text{id}_\GZ{1}} & \GZ{2} \arrow[hook]{r}{\iota^*_{2}} \arrow[hook]{d}{\text{id}_\GZ{2}} & \cdots \arrow[hook]{r}{\iota^*_{n-1}} & \GZ{n} \arrow[hook]{r}{\iota^*_{n}} \arrow[hook]{d}{\text{id}_\GZ{n-1}} & \cdots \\
        0 \arrow[hook]{r}{\iota^*_{0}} & U(\gl{1}) \arrow[hook]{r}{\iota^*_{1}} & U(\gl{2}) \arrow[hook]{r}{\iota^*_{2}} & \cdots \arrow[hook]{r}{\iota^*_{n-1}} & U(\gl{n}) \arrow[hook]{r}{\iota^*_{n}} & \cdots,
    \end{tikzcd}
    \label{equation:4.6}
\end{equation}

Using the upper chain of embeddings of algebras $\GZ{n}$ of the diagram (\ref{equation:4.6}), it is possible to correctly define the infinite Gelfand--Tsetlin algebra $\GZi$. So, we will understand the copy of the upper chain as the algebra $\GZi=\varinjlim\GZ{n}$. It is easy to see that the Gelfand--Tsetlin algebra $\GZi$ is generated by all coefficients of all quantum minors $A_m(u)$, that is, $\GZi=\langle\{a_{m i}\}_{1\leq i\leq m\leq\infty}\rangle$. Thus, the spectral problem (\ref{equation:4.5}) can also be put on a basis in irreducible $\gli$ modules. Below, using the infinite Gelfand-Cetlin algebra $\GZi$, we, by analogy with the algebra $\gl{n}$, define the Gelfand--Tsetlin basis in irreducible representations of the algebra $\gli$.

        \section{Polynomial representations of infinite-dimensional general linear Lie algebra}
        Let's proceed to the consideration of the algebra $\gli$ that interests us. First of all, we will limit the class of irreducible representations we are considering. Namely, similarly to the finite-dimensional case, we define the highest weight as an ordered set $\lambda = (\lambda_1,\lambda_2,\ldots)$ such that
\begin{equation}
    \lambda_1 \geq \lambda_2 \geq \ldots \qquad \text{and} \qquad  \forall i \in \mathbb{N}^* \implies \lambda_i \in \mathbb{N},
    \label{equation:5.1}
\end{equation}
where there exists a number $N \in \mathbb{N}^*$, such that for any $n > N$ all $\lambda_n = 0$.

As in the case of $\gl{n}$, we identify all the highest weights with the corresponding Young diagrams. Denote by $\Complex^\infty$ the direct sum of all one-dimensional subspaces spanned by vectors $e_i$, in other words $\Complex^\infty =\bigoplus\limits_{i\in\mathbb{N}^*} \langle e_i\rangle$. Vectors $e_i$ can be perceived as columns of dimension $\infty\times 1$, in which there are zeros in each place, and one in the $i$-th place. In other words, $e_i = (\delta_{i j})_{j\in\mathbb{N}^*}$. Let's define the action of the algebra $\gli$ on the space $\Complex^\infty$ by the usual multiplication of a vector by a matrix, that is, $E_{i j}\cdot e_k = \delta_{j k} e_i$. The representation $\Complex^\infty$ will be called the tautological representation of the algebra $\gli$.

Recall that, in the case of the algebra $\gl{n}$, any irreducible polynomial representation of it was a solution to a universal problem (\ref{equation:2.8}). This observation motivates us to define irreducible polynomial representations $V^\lambda$ of the algebra $\gli$ as solutions to a similar universal problem. Namely, we will call the representation $V^\lambda$ of the algebra $\gli$ polynomial if there is a universal symmetrizing map $\Phi_{uni} : (\Complex^\infty)^{\times\lambda}\xrightarrow{\quad} V^\lambda$. That is, such a map that for any symmetrizing map $\Phi : (\Complex^\infty )^{\times\lambda}\xrightarrow{\quad} F$, there exists a single map $\varphi : V^\lambda\xrightarrow{\quad} F$, such that $\Phi = \varphi \circ \Phi_{uni}$. In other words, for any $\Phi$ and $F$ as above, there is a single map $\varphi$ that makes the following diagram commutative
\begin{equation}
    \begin{tikzcd}[column sep = 4em, row sep = 4em]
        ( \Complex^\infty )^{\times \lambda} \arrow{r}{\Phi_{uni}} \arrow[swap]{dr}{\Phi} & V^\lambda \arrow[dashrightarrow]{d}{\varphi} \\
        & F.
        \label{equation:5.2}
    \end{tikzcd}
\end{equation}

Generally speaking, we have not yet defined the action of the algebra $\gli$ on the spaces $V^\lambda$, therefore, formally speaking, we do not have the right to call them $\gli$-modules. However, despite the possible confusion, we will continue to adhere to the terminology introduced. So, having defined the representations $V^\lambda$, first of all we will show their existence and uniqueness (up to isomorphism), this will be the result of the following proposition.

\begin{proposition}
    For any highest weight $\lambda$ representations $V^\lambda$ exist and are unique up to isomorphism.
\end{proposition}

\begin{proof}
    Let's fix the highest weight $\lambda = \lambda_1 \geq\ldots \geq\lambda_N \geq 0 \geq\ldots$ of the representation $V^\lambda$. Let's call a map $f$ semisymmetrizing if it satisfies only the first two conditions (\ref{equation:2.4}) and (\ref{equation:2.5}). Consider a similar universal problem with semisymmetrizing maps: if $\Phi_{uni}' : (\Complex^\infty)^{\times\lambda}\xrightarrow{\quad} W^\lambda$ is a semisymmetrizing map, such that for any semisymmetrizing map $\Phi : ( \Complex^\infty)^{\times\lambda} \xrightarrow{\quad} F$ there is a single map $\varphi' : W^\lambda\xrightarrow{\quad} F$, such that $\Phi=\varphi'\circ\Phi_{uni}'$. In other words, for any $\Phi$ and $F$ as above, there is a single map $\varphi'$ that makes the following diagram commutative
    \begin{equation}
        \begin{tikzcd}[column sep = 4em, row sep = 4em]
            ( \Complex^\infty )^{\times \lambda} \arrow{r}{\Phi_{uni}'} \arrow[swap]{dr}{\Phi} & W^\lambda \arrow[dashrightarrow]{d}{\varphi'} \\
            & F.
            \label{equation:5.3}
        \end{tikzcd}
    \end{equation}
    
    From the universal properties of the tensor product, it is obvious that the solution will be the space $W^\lambda = \bigotimes\limits_{k=1}^{N} \bigwedge\nolimits^{\lambda_k} \Complex^\infty$, and the universal map $\Phi_{uni}'$ is given by the formula
    \begin{equation}
        \Phi_{uni}':
        \begin{ytableau}
            v_1^{(1)} & v_1^{(2)} & \none[\dots] & v_1^{(N)} \\
            v_2^{(1)} & v_2^{(2)} & \none[\dots] & v_{\lambda_N}^{(N)} \\
            \none[\vdots] & \none[\vdots] \\
            v_{\lambda_1}^{(1)} & v_{\lambda_2}^{(2)} \\
        \end{ytableau}
        \xmapsto{\quad}
        \bigotimes\limits_{k = 1}^N \bigwedge\limits_{i = 1}^{\lambda_k} v_i^{(k)}.
        \label{equation:5.4}
    \end{equation}

    Let's define the subspace $U^\lambda$ of the space $W^\lambda$ as the subspace generated by all the differences $\Phi_{uni}'(w) -\sum\Phi_{uni}'(w')$, where summation is implied by those $w'\in (\Complex^\infty )^{\times\lambda}$, resulting from $w$ by exchange between two fixed columns, with the selected subset in the right of the selected columns (see \ref{equation:2.7}). As is known (see \cite{mac2013categories} Ch.3), the factorspace $W^\lambda/U^\lambda$ has the following universal property: for the projection $\pi^\lambda: W^\lambda\xrightarrow{\quad} W^\lambda/U^\lambda : w\mapsto[w]$, which sends the element $w$ into its conjugacy class $[w]$, and any linear map $\varphi' : W^\lambda\xrightarrow{\quad}F$ there is a single linear map $\varphi : W^\lambda/U^\lambda \xrightarrow{\quad} F$, such that $\varphi' = \varphi \circ \pi^\lambda$. In other words, for any $\varphi'$ and $F$ as above, there is a single $\varphi$ map that makes the following diagram commutative
    \begin{equation}
        \begin{tikzcd}[column sep = 4em, row sep = 4em]
            W^\lambda \arrow{r}{\pi^\lambda} \arrow[swap]{dr}{\varphi'} & W^\lambda/U^\lambda \arrow[dashrightarrow]{d}{\varphi} \\
            & F.
            \label{equation:5.5}
        \end{tikzcd}
    \end{equation}

    So, commutative diagrams (\ref{equation:5.3}) and (\ref{equation:5.5}) can be completed to the diagram (\ref{equation:5.2}) with symmetrizing maps. In other words, the commutativity of diagrams (\ref{equation:5.3}) and (\ref{equation:5.5}) induces the commutativity of the following diagram 
    \begin{equation}
        \begin{tikzcd}[column sep = 4em, row sep = 4em]
            ( \Complex^\infty )^{\times \lambda} \arrow{r}{\Phi_{uni}'} \arrow[swap]{dr}{\Phi} \arrow[bend left]{rr}{\Phi_{uni}} & W^\lambda \arrow{r}{\pi^\lambda} \arrow[dashrightarrow]{d}{\varphi'} & W^\lambda/U^\lambda \arrow[dashrightarrow]{dl}{\varphi} \\
            & F,
            \label{equation:5.6}
        \end{tikzcd}
    \end{equation}
    where $\Phi_{uni} = \pi^\lambda \circ \Phi_{uni}'$.

    Thus, the commutative diagram (\ref{equation:5.6}) gives a solution to the original problem, thereby proving the existence of polynomial representations $V^\lambda$. Namely, for each polynomial representation $V^\lambda$, there is an isomorphism of vector spaces
    \begin{equation}
        V^\lambda \cong \Big( \bigotimes\limits_{k=1}^{N} \bigwedge\nolimits^{\lambda_k} \Complex^\infty \Big) / U^\lambda.
        \label{equation:5.7}
    \end{equation}

    Finally, in accordance with this formula, we define the action of the algebra $\gli$ on the space $V^\lambda$, which justifies the name "polynomial representation". The action of the element $g\in\gli$ is given as the usual action of the Lie algebra on the factor space and tensor product, and the action on the space $\Complex^\infty$ is given as the action of a tautological representation. In other words, the following formula holds
    \begin{equation}
        g \cdot [~v_1^{(1)} \wedge v_2^{(1)} \wedge \ldots~] = [~g \cdot v_1^{(1)} \wedge v_2^{(1)} \wedge \ldots + v_1^{(1)} \wedge g \cdot v_2^{(1)} \wedge \ldots + \ldots~].
        \label{equation:5.8}
    \end{equation}

    Now, let us prove the uniqueness up to isomorphism. Suppose the contrary: let there exist a polynomial representation $\widetilde{V}^\lambda$ for the same highest weight $\lambda$, which is not isomorphic to $V^\lambda$. Then $\widetilde{V}^\lambda$ is also a solution to the universal problem (\ref{equation:5.2}), and for any $\widetilde{\Phi}$ and $F$, the following diagram is commutative
    \begin{equation}
        \begin{tikzcd}[column sep = 4em, row sep = 4em]
            ( \Complex^\infty )^{\times \lambda} \arrow{r}{\widetilde{\Phi}_{uni}} \arrow[swap]{dr}{\widetilde{\Phi}} & \widetilde{V}^\lambda \arrow[dashrightarrow]{d}{\widetilde{\varphi}} \\
            & F.
            \label{equation:5.9}
        \end{tikzcd}
    \end{equation}

    Let $\widetilde{\Phi} =\Phi_{uni}$ and $F=V^\lambda$, then, by the universal property of $V^\lambda$, there is a single symmetrizing map of $\varphi : V^\lambda\xrightarrow{\quad} \widetilde{V}^\lambda$, making the following diagram commutative
    \begin{equation}
        \begin{tikzcd}[column sep = 4em, row sep = 4em]
            ( \Complex^\infty )^{\times \lambda} \arrow{r}{\widetilde{\Phi}_{uni}} \arrow[swap]{dr}{\Phi_{uni}} & \widetilde{V}^\lambda \arrow[dashrightarrow, shift left=0.5ex]{d}{\widetilde{\varphi}} \\
            & V^\lambda \arrow[dashrightarrow, shift left=0.5ex]{u}{\varphi}.
            \label{equation:5.10}
        \end{tikzcd}
    \end{equation}

    Note that $\widetilde{\varphi} \circ\varphi \circ\Phi_{uni} =\widetilde{\varphi} \circ\widetilde{\Phi}_{uni} =\Phi_{uni}=\text{id}_{V^\lambda} \circ\Phi_{uni}$, but the map $\widetilde{\varphi} \circ\varphi$ is unique, which means it is equal to the identity $\widetilde{\varphi} \circ\varphi = \text{id}_{V^\lambda}$. Similar reasoning shows that $\varphi\circ\widetilde{\varphi} = \text{id}_{\widetilde{V}^\lambda}$. Thus, $\widetilde{V}^\lambda \cong V^\lambda$. We got a contradiction, which means the solution of a universal problem (\ref{equation:5.2}) is the only one up to isomorphism.
\end{proof}

Thus, all polynomial representations of the algebra $\gli$ are correctly defined. Next, it will be shown that all polynomial representations of the algebra $\gli$ are irreducible and have a highest weight vector. It will also be shown that different highest weights define different (up to isomorphism) polynomial representations.
    
        \section{The Gelfand--Tsetlin basis in irreducible polynomial representations}
        Let's proceed to the construction of the Gelfand--Tsetlin basis in irreducible $\gli$-modules. As mentioned earlier, we define the Gelfand--Tsetlin basis in the irreducible representation $V$ similarly to the finite-dimensional case. Namely, we will call the basis $\{e_T\}$ of the irreducible representation $V$ of the algebra $\gli$ the Gelfand--Tsetlin basis if the Gelfand--Tsetlin algebra $\GZi$ acts diagonally on it. Note that we do not yet know that the polynomial representations $V^\lambda$ are irreducible. To prove this fact, we will construct a certain basis in the representation $V^\lambda$, which will soon turn out to be the desired Gelfand--Tsetlin basis. However, before that, we will define infinite Gelfand--Tsetlin schemes, which will index the elements of our basis. For a fixed highest weight $\lambda = \lambda_1\geq\lambda_2\geq\ldots$, the infinite Gelfand--Tsetlin scheme is an ordered set of elements $\lambda_{i j}$
\begin{equation}
    \Lambda = \left[
    \begin{array}{ccccccccc}
    \cdots & & \cdots & & \cdots & & \cdots & & \cdots \\
    & \lambda_{n,1} & & \lambda_{n,2} & & \cdots & & \lambda_{n, n} & \\
    & & \cdots & & \cdots & & \cdots & & \\
    & & & \lambda_{2 1} & & \lambda_{2 2} & & & \\
    & & & & \lambda_{1 1} & & & & \\
    \end{array}
    \right],
    \label{equation:6.1}
\end{equation}
where for any $1\leq j\leq i\leq\infty$ the elements of $\lambda_{i j}$ are defined as the number of cells less or equal to $i$ in the $j$-th row of some semi-standard Young tableau of the form $\lambda$ (see \cite{fulton1997young}).

For the polynomial representation $V^\lambda$ with the highest weight $\lambda =\lambda_1\geq\ldots\geq\lambda_N\geq 0\geq\ldots$, using the infinite Gelfand--Tsetlin schemes defined above, we define the elements $e_\Lambda\in V^\lambda$ by the formula
\begin{equation}
    e_\Lambda = \prod\limits_{2 \leq k < \infty}^{\longrightarrow} \prod\limits_{i = 1}^{k-1} z_{k i}^{\lambda_{k i} - \lambda_{k-1, i}} \cdot v_\lambda,
    \label{equation:6.2}
\end{equation}
where the terms in the ordered product are ordered according to the increasing index $k$, and the vector $v_\lambda\in V^\lambda$ has the form
\begin{equation}
    v_\lambda = [~(e_1 \wedge e_2 \wedge \ldots \wedge e_{\lambda_1}) \otimes (e_1 \wedge e_2 \wedge \ldots \wedge e_{\lambda_2}) \otimes \ldots \otimes (e_1 \wedge e_2 \wedge \ldots \wedge e_{\lambda_N})~].
    \label{equation:6.3}
\end{equation}

Note that on the right side of the expression (\ref{equation:6.2}) only a finite number of differences $\lambda_{k i} - \lambda_{k-1, i}$ are not equal to zero. Therefore, in reality, the ordered product has only a finite number of terms, which means that the action on the vector $v_\lambda$ is defined correctly. For the elements $e_\Lambda$, we formulate and prove the following lemma.
\begin{lemma}
    For the polynomial representation $V^\lambda$, the family of vectors $\{e_\Lambda\}_{\Lambda\in\Lambda_\infty(\lambda)}$ indexed by all infinite Gelfand--Tsetlin schemes forms the basis of the vector space $V^\lambda$.
\end{lemma}
\begin{proof}
    Let $\mathfrak{h}^n\subset\gl{n}$ be the Cartan subalgebra, $E_{i j}$ be the matrix unit, $\text{deg} E_{i j} = j - i$ be the degree of the matrix unit and $\mathfrak{g}_q =\langle \{E_{i j}~|~\text{deg} E_{i j} = q \} \rangle$ -- span of matrix units of degree $q$ over $\Complex$. For algebra $\gl{n}$ we have a triangular decomposition $\gl{n}=\mathfrak{n}^n_-\oplus\mathfrak{h}^n\oplus\mathfrak{n}^n_+$, where $\mathfrak{n}^n_- = \bigoplus \limits_{n > q > 0} \mathfrak{g}_q$, and $\mathfrak{n}^n_+ = \bigoplus \limits_{n < q < 0} \mathfrak{g}_q$. We introduce a similar decomposition for the algebra $\gli$, namely, we define the Cartan subalgebra $\mathfrak{h}^\infty=\mathfrak{g}_0$ and the subspaces $\mathfrak{n}^\infty_-=\bigoplus\limits_{q>0}\mathfrak{g}_q$ and $\mathfrak{n}^\infty_+ = \bigoplus \limits_{q < 0} \mathfrak{g}_q$. Then there is a triangular decomposition $\gli =\mathfrak{n}^\infty_-\oplus\mathfrak{h}^\infty\oplus \mathfrak{n}^\infty_+$. For the highest weight $\lambda$, we define representations $V^\lambda_n=\mathfrak{n}^{n}_-\cdot v_\lambda$ of the algebra $\gl{n}$. Now let's go back to the diagram (\ref{equation:3.8}). Note that under the embedding $d_e \iota_{n-1}$, the subspace $\mathfrak{n}^{n-1}_-$ is embedded into $\mathfrak{n}^{n}_-$, that is, $d_e \iota_{n-1}(\mathfrak{n}^{n-1}_-) \subset \mathfrak{n}^{n}_-$. This observation induces the following diagram of embeddings
    \begin{equation}
        \begin{tikzcd}[column sep = 4em, row sep = 4em]
            0  \arrow[hook]{r}{d_e \iota_0} \arrow[d, squiggly] & \mathfrak{n}^{1}_- \arrow[hook]{r}{d_e \iota_1} \arrow[d, squiggly] & \mathfrak{n}^{2}_- \arrow[hook]{r}{d_e \iota_2} \arrow[d, squiggly] & \cdots \arrow[hook]{r}{d_e \iota_{n-1}} & \mathfrak{n}^{n}_- \arrow[hook]{r}{d_e \iota_n} \arrow[d, squiggly] & \cdots \\
            V^\lambda_0 \arrow[hook]{r}{\#\iota_0} & V^\lambda_1 \arrow[hook]{r}{\#\iota_1} & V^\lambda_2 \arrow[hook]{r}{\#\iota_2} & \cdots \arrow[hook]{r}{\#\iota_{n-1}} & V^\lambda_n \arrow[hook]{r}{\#\iota_n} & \cdots,
        \end{tikzcd}
        \label{equation:6.4}
    \end{equation}
    where $\#\iota_n$ is some embedding that sends highest weight vector $v_\lambda\in V^\lambda_n$ into the highest weight vector $v_\lambda\in V^\lambda_{n+1}$.

    The lower chain of embeddings of spaces $V^\lambda_n$ forms a directed system and has a colimit $\varinjlim V^\lambda_n= V^\lambda$. Hence, in particular, it can be seen that the polynomial representation $V^\lambda$ is generated by the action of the subspace $\mathfrak{n}^{\infty}_-$ on the vector $v_\lambda$, that is, $V^\lambda=\mathfrak{n}^{\infty}_- \cdot v_\lambda$. Let's define the degree $\text{deg}\Lambda$ of the Gelfand--Tsetlin scheme $\Lambda$ as the smallest number $N$, such that for any $k>N$ all differences $\lambda_{k i} - \lambda_{k-1, i}$ of the elements of the Gelfand--Tsetlin scheme $\Lambda$ are zero for any $i$. Let's introduce subset $\Lambda^n_\infty(\lambda)\subset\Lambda_\infty(\lambda)$ as a set consisting of all Gelfand--Tsetlin schemes of degree $n$, in other words, $\Lambda^n_\infty(\lambda) = \{ \Lambda \in \Lambda_\infty(\lambda)~|~\text{deg} \Lambda = n \}$. Then there is equality
    \begin{equation}
        \Lambda_\infty(\lambda) = \bigcup\limits_{n > 0} \Lambda^n_\infty(\lambda).
        \label{equation:6.5}
    \end{equation}

    As mentioned above, in the product (\ref{equation:6.2}) there are only a finite number of terms. More formally, if $\Lambda\in\Lambda^n_\infty(\lambda)$, then $e_\Lambda =\prod\limits_{2\leq k <\infty}^{\longrightarrow} \prod\limits_{i =1}^{k-1} z_{k i}^{\lambda_{k i} - \lambda_{k-1, i}} \cdot v_\lambda = \prod\limits_{2 \leq k \leq n}^{\longrightarrow} \prod\limits_{i = 1}^{k-1} z_{k i}^{\lambda_{k i} - \lambda_{k-1, i}} \cdot v_\lambda$. However, for all representations $V^\lambda_n$ of the algebra $\gl{n}$ set $\{e_\Lambda\}_{\Lambda\in\Lambda^n_\infty(\lambda)}$ forms the basis, and therefore the set $\{e_\Lambda\}_{\Lambda\in\Lambda_\infty(\lambda)} = \bigcup\limits_{n > 0} \{ e_\Lambda \}_{\Lambda \in \Lambda^n_\infty(\lambda)}$ forms the basis in the polynomial representation $V^\lambda$.
\end{proof}
Now we prove that the basis $\{e_\Lambda\}_{\Lambda\in\Lambda_\infty(\lambda)}$ of the polynomial representation $V^\lambda$ is the desired Gelfand--Tsetlin basis. However, before doing this, it must be shown that the polynomial representations $V^\lambda$ are irreducible and are defined correctly. This will be the result of the following lemmas.
\begin{lemma}
    The polynomial representation $V^\lambda$ of the algebra $\gli$ is an irreducible representation with the highest weight.
\end{lemma}
\begin{proof}
    To begin with, we show that the vector $v_\lambda \in V^\lambda$ is the highest weight vector. The fact that $\mathfrak{n}^{\infty}_- \cdot v_\lambda=V^\lambda$ was proved in the previous lemma. Let's show the validity of the remaining properties. Note that there is an equality $\mathfrak{n}^{\infty}_+ = \bigcup\limits_{n > 0} \mathfrak{n}^{n}_+$. But, for any $n\in\mathbb{N^*}$, the action of $\mathfrak{n}^{n}_+$ on the vector $v_\lambda$ gives a null vector, which means that $\mathfrak{n}^{\infty}_+\cdot v_\lambda = 0$. Similarly, the equality $\mathfrak{h}^{\infty} = \bigcup\limits_{n > 0} \mathfrak{h}^{n}$ holds. But, for any $n\in\mathbb{N^*}$, the action of the element $E_{ii}\in\mathfrak{h}^{n}$ on the vector $v_\lambda$ gives $\lambda_i v_\lambda$, which means $\forall E_{ii} \in \mathfrak{h}^{\infty} \implies E_{ii} \cdot v_\lambda = \lambda_i v_\lambda$. Thus, all polynomial representations are representations with the highest weight. Let's show their irreducibility. 
    
    Let's define the positive definite Hermitian form $\langle\cdot|\cdot\rangle$ on the space $V^\lambda$, declaring the basis elements $e_\Lambda$ orthogonal. Let $U \subset V^\lambda$ be a nontrivial subspace invariant with respect to the action of the algebra $\gli$. Consider the orthogonal complement $U^\perp$ with respect to the form $\langle\cdot |\cdot\rangle$. It is clear that the subspace $U^\perp$ is also invariant with respect to the action of the algebra $\gli$. Indeed, for any $g\in\gli$ it follows that $g\cdot U\subset U$. But, by definition of the orthogonal complement, it is true that $\langle U |U^\perp\rangle = 0$. Then $\langle g\cdot U|U^\perp\rangle =\langle U|g^\dagger\cdot U^\perp\rangle = 0$, where $g^\dagger$ is a Hermitian conjugate matrix. Therefore, $g^\dagger\cdot U^\perp\subset U^\perp$, which means $U^\perp$ is an invariant subspace. So, there is a decomposition into irreducible representations $V^\lambda = U \oplus U^\perp$. Let the highest weight vector $v_\lambda$ belong to the subspace $U$. The action of the algebra $\gli\cdot v_\lambda = V^\lambda$. On the other hand, $U$ is an invariant subspace, which means $\gli\cdot v_\lambda\subset U$. Therefore, then the following equality holds $U=V^\lambda$, and $U^\perp = 0$.
\end{proof}
\begin{lemma}
    Polynomial representations with different highest weights are not isomorphic.
\end{lemma}
\begin{proof}
    Let's fix the polynomial representations $V^\lambda$ and $V^{\lambda'}$ such that $\lambda\neq\lambda'$. Suppose the opposite, let $\rho : V^\lambda\xrightarrow{\quad} V^{\lambda'}$ be an isomorphism of representations. With isomorphism, the highest weight vector $v_\lambda$ goes to the highest weight vector $v_{\lambda'}$. Then for any $E_{i i} \in \mathfrak{h}^\infty$ it is true that $E_{i i}\cdot v_\lambda =\lambda_i v_\lambda$. By the property of the intertwining operator $\rho$, we obtain that $\lambda_i v_\lambda = \rho (E_{i i}\cdot v_\lambda) = E_{i i} \cdot \rho (v_\lambda) = E_{i i} \cdot v_{\lambda'} = \lambda'_i \cdot v_{\lambda'}$ for any index $i$. We got a contradiction, which means that polynomial representations with different higher weights are not isomorphic.
\end{proof}

So, all polynomial representations of the algebra $\gli$ are irreducible representations with highest weight, and there is a bijection between the classes of isomorphisms of irreducible representations and the set of all highest weights. All these properties were inherited from similar finite-dimensional representations.
\begin{theorem}
    Basis $\{e_\Lambda \}_{\Lambda \in \Lambda_\infty(\lambda)}$ of the irreducible polynomial representation $V^\lambda$ is the Gelfand--Tsetlin basis.
\end{theorem}
\begin{proof}
    As mentioned earlier, the infinite Gelfand--Tsetlin algebra $\GZi$ is generated by all coefficients of all quantum minors $A_m(u)$, that is, $\GZi=\langle\{a_{m i}\}_{1\leq i\leq m\leq\infty}\rangle$. Thus, it is sufficient to show that all vectors $e_\Lambda$ are eigenvectors of quantum minors $A_m(u)$. Note that the equality $\{A_m(u)\}_{1 \leq m <\infty} = \bigcup\limits_{n>0} \{A_m(u)\}_{1\leq m\leq n}$ holds. Due to the formula (\ref{equation:4.5}), for any $n\in\mathbb{N}$, all elements of the set $\{e_\Lambda\}_{\Lambda\in \Lambda^n_\infty(\lambda)}$ are the eigenvectors of quantum minors $\{A_m(u)\}_{1\leq m\leq n}$. Therefore, the basis is $\{e_\Lambda\}_{\Lambda\in\Lambda_\infty(\lambda)}$ is the set of eigenvectors of all quantum minors $\{A_m(u)\}_{1\leq m<\infty}$, which means it is the Gelfand--Tsetlin basis.
\end{proof}

Gelfand--Tsetlin basis (\ref{equation:6.2}) has the simplest form in the fundamental representations of the algebra $\gli$. Similarly to the case of $\gl{n}$, for any polynomial representation $V^\lambda$ of the algebra $\gli$ there is an embedding
\begin{equation}
    V^\lambda \xhookrightarrow{\quad} \bigotimes\limits_{k \geq 0} \Sym^{a_k} \bigwedge\nolimits^{k} \Complex^\infty,
    \label{equation:6.6}
\end{equation}
where almost all the numbers $a_k\in \mathbb{N}$ are zeros.

So, for the fundamental representation $\bigwedge\nolimits^{k}\Complex^\infty$, the Gelfand--Tsetlin basis has the form
\begin{equation}
    \{ e_\Lambda \}_{\Lambda \in \Lambda_\infty(\lambda)} = \{ e_{i_1} \wedge e_{i_2} \wedge \ldots \wedge e_{i_k}\}_{1 \leq i_1 \leq i_2 \leq \ldots \leq i_k < \infty}.
    \label{equation:6.7}
\end{equation}

As noted earlier, all representations of the algebra $\gli$ are also representations of the group $\GLi$. Thus, the basis (\ref{equation:6.2}) is the Gelfand--Tsetlin basis of all polynomial representations $V^\lambda$ of the group $\GLi$. As can be seen from the above, all the ideas of constructing the Gelfand--Tsetlin basis in irreducible $\gli$ modules were based on reducing the infinite-dimensional case to a finite-dimensional one. These ideas can be applied to construct the Gelfand--Tsetlin basis in irreducible representations of the colimits of all classical Lie algebras $A_\infty, B_\infty, C_\infty, D_\infty$ and other reductive algebras.

        \bibliographystyle{alpha}
        \bibliography{bibliography}
        \hspace{1cm}
    \end{sloppypar}
\end{document}